\documentclass[11pt,oneside]{amsart}
\PassOptionsToPackage{nameinlink,capitalise}{cleveref}
\usepackage[T1]{fontenc}
\usepackage[latin9]{inputenc}
\usepackage{color}
\usepackage{mathrsfs}
\usepackage{mathtools}
\usepackage{url}
\usepackage{amsbsy}
\usepackage{amstext}
\usepackage{amsthm}
\usepackage{amssymb}
\usepackage{geometry}
\usepackage[bookmarks=true,bookmarksnumbered=false,bookmarksopen=false,
 breaklinks=false,pdfborder={0 0 1},backref=false,colorlinks=true]
 {hyperref}
\hypersetup{
 allcolors={[rgb]{0,0.2,0.6}}}

\makeatletter
\numberwithin{equation}{section}
\numberwithin{figure}{section}

\usepackage{shuffle}
\usepackage{crossreftools}

\AddToHook{begindocument/end}{%
  \crefname{thm}{Theorem}{Theorems}%
  \Crefname{thm}{Theorem}{Theorems}%

  \crefname{lem}{Lemma}{Lemmas}%
  \Crefname{lem}{Lemma}{Lemmas}%

  \crefname{prop}{Proposition}{Propositions}%
  \Crefname{prop}{Proposition}{Propositions}%

  \crefname{cor}{Corollary}{Corollaries}%
  \Crefname{cor}{Corollary}{Corollaries}%

  \crefname{defn}{Definition}{Definitions}%
  \Crefname{defn}{Definition}{Definitions}%
}

\makeatother

\theoremstyle{plain}
\newtheorem{thm}{\protect\theoremname}
\newtheorem{conjecture}[thm]{\protect\conjecturename}
\newtheorem{cor}[thm]{\protect\corollaryname}
\newtheorem{prop}[thm]{\protect\propositionname}
\theoremstyle{definition}
\newtheorem{defn}[thm]{\protect\definitionname}
\theoremstyle{plain}
\newtheorem{lem}[thm]{\protect\lemmaname}

\usepackage{cleveref}
\AddToHook{env/conjecture/begin}{\crefalias{thm}{conjecture}}
\AddToHook{env/cor/begin}{\crefalias{thm}{cor}}
\AddToHook{env/prop/begin}{\crefalias{thm}{prop}}
\AddToHook{env/defn/begin}{\crefalias{thm}{defn}}
\AddToHook{env/lem/begin}{\crefalias{thm}{lem}}
\crefname{conjecture}{conjecture}{conjectures}
\Crefname{conjecture}{Conjecture}{Conjectures}
\crefname{cor}{corollary}{corollaries}
\Crefname{cor}{Corollary}{Corollaries}
\crefname{defn}{definition}{definitions}
\Crefname{defn}{Definition}{Definitions}
\crefname{lem}{lemma}{lemmas}
\Crefname{lem}{Lemma}{Lemmas}
\crefname{prop}{proposition}{propositions}
\Crefname{prop}{Proposition}{Propositions}
\crefname{thm}{theorem}{theorems}
\Crefname{thm}{Theorem}{Theorems}
\providecommand{\conjecturename}{Conjecture}
\providecommand{\corollaryname}{Corollary}
\providecommand{\definitionname}{Definition}
\providecommand{\lemmaname}{Lemma}
\providecommand{\propositionname}{Proposition}
\providecommand{\theoremname}{Theorem}

\begin{document}
\address[Minoru Hirose]{Graduate School of Science and Engineering, Kagoshima University, 1-21-35 Korimoto, Kagoshima, Kagoshima 890-0065, Japan}
\email{hirose@sci.kagoshima-u.ac.jp}
\title[Confluence relations, Hoffman basis over $\mathbb{Z}_{(2)}$, and
the DD conjecture]{Confluence relations, Hoffman basis over $\mathbb{Z}_{(2)}$, and
the Deligne--Drinfeld conjecture}
\author{Minoru Hirose}
\begin{abstract}
We express every admissible word as a $\mathbb{Z}_{(2)}$-linear combination
of Hoffman words and explicit confluence relations. This implies that
every multiple zeta value is a linear combination of elements of the
Hoffman basis with coefficients having odd denominators, as conjectured
by Hirose and Sato. Combined with Brown's theorem and Furusho's characterization
of the Grothendieck--Teichmüller group $\mathrm{GRT}$ by confluence
relations, our result identifies the de Rham motivic Galois group
of mixed Tate motives over $\mathbb{Z}$ with $\mathrm{GRT}$ and
gives an alternative proof of the Deligne--Drinfeld conjecture, for
which a proof was previously presented in the OpenAI Math Release.
\end{abstract}

\maketitle
\global\long\def\Z{\mathbb{Z}}%
 
\global\long\def\Ibb{\mathbb{I}}%
 
\global\long\def\R{\mathbb{R}}%
  
\global\long\def\A{\mathcal{A}}%
 
\global\long\def\Wcal{\mathcal{W}}%
 
\global\long\def\Zcal{\mathcal{Z}}%
 
\global\long\def\Mcal{\mathcal{M}}%
\global\long\def\Hcal{\mathcal{H}}%
 
\global\long\def\sh{\shuffle}%
 
\global\long\def\id{\mathrm{id}}%
 
\global\long\def\Const{\operatorname{Const}}%
  
\global\long\def\odd{\operatorname{o}}%
 
\global\long\def\depth{\operatorname{depth}}%
   
\global\long\def\eps{\varepsilon}%
 
\global\long\def\ind#1{\boldsymbol{#1}}%
 
\global\long\def\br#1{\mathbf{1}_{\{#1\}}}%
 
\global\long\def\word{\mathsf{w}}%
 
\global\long\def\kk{\ind k}%
 
\global\long\def\bb{\ind b}%
 
\global\long\def\avec{\ind a}%

\tableofcontents{}

\section{Introduction\label{sec:introduction}}

An index $\kk=(k_{1},\ldots,k_{d})$ of positive integers is called
admissible if $d=0$ or $k_{d}\ge2$. Its multiple zeta value (MZV)
is defined by
\[
\zeta(k_{1},\ldots,k_{d})\coloneqq\sum_{0<n_{1}<\cdots<n_{d}}\frac{1}{n^{k_{1}}_{1}\cdots n^{k_{d}}_{d}},\qquad\zeta(\varnothing)\coloneqq1.
\]
Its weight is $|\kk|\coloneqq k_{1}+\cdots+k_{d}$. An index whose
entries all belong to $\{2,3\}$ is called a Hoffman index. For $N\ge0$,
put
\begin{align*}
\Ibb_{N} & \coloneqq\{\kk:\kk\text{ is admissible},\ |\kk|=N\},\\
\Ibb^{\mathrm{H}}_{N} & \coloneqq\{\kk\in\Ibb_{N}:\kk\text{ is a Hoffman index}\}.
\end{align*}
In particular, $\Ibb_{0}=\Ibb^{\mathrm{H}}_{0}=\{\varnothing\}$.

Let $\Zcal_{N}$ be the $\mathbb{Q}$-vector space spanned by MZVs
of weight $N$, and put $d_{N}\coloneqq\#\Ibb^{\mathrm{H}}_{N}$.
The following theorems were proved using the theory of mixed Tate
motives.
\begin{thm}[{Deligne--Goncharov \cite[Corollary 5.25]{DeligneGoncharov}; Terasoma
\cite{TerasomaMT}}]
For every $N\ge0$, we have
\[
\dim_{\mathbb{Q}}\Zcal_{N}\le d_{N}.
\]
\end{thm}

\begin{thm}[{Brown \cite[Theorem 1.1]{BrownMT}}]
Every MZV is a $\mathbb{Q}$-linear combination of MZVs with Hoffman
indices of the same weight.
\end{thm}

Write $\Z_{(2)}$ for the ring of rational numbers with odd denominators.
\begin{conjecture}[{Hirose--Sato \cite[Remark 15]{HSblock}}]
\label{conj:hirose-sato}Every MZV is a $\Z_{(2)}$-linear combination
of MZVs with Hoffman indices of the same weight.
\end{conjecture}

In this paper, we give a direct algebraic proof of this conjecture
using an explicit family of confluence relations, without relying
on motivic methods. In fact, we prove the corresponding generation
statement modulo confluence relations.

Related arithmetic reductions are known for suitably normalized Euler
sums \cite[Theorem 3.10]{HSEuler} and for expansions with integer
coefficients in MZVs whose entries are all at least $2$ \cite[Theorems 1.7 and 1.8]{HMSWdrop1}.

\subsection{Main results\label{sec:statement}}

Let $\A\coloneqq\Z\langle e_{0},e_{1}\rangle$ be the noncommutative
polynomial ring, and put 
\[
\A^{0}\coloneqq\Z\oplus e_{1}\A e_{0}.
\]
Unless otherwise stated, multiplication of words means concatenation.
The empty word is denoted by $1$. For a word $w$, write $|w|$ for
its length. The correspondence between admissible indices and words
is 
\[
\word(\kk)\coloneqq e_{1}e^{k_{1}-1}_{0}\cdots e_{1}e^{k_{d}-1}_{0},\qquad\word(\varnothing)\coloneqq1.
\]

Let $\mathcal{I}_{\mathrm{CF}}\subset\A^{0}$ be the $\Z$-submodule
generated by the confluence relations \cite[Definition 22 and Remark 23]{HSconfluence},
recalled in \cref{sec:confluence}. Define 
\[
\Zcal^{\mathrm{cf}}\coloneqq(\A^{0}/\mathcal{I}_{\mathrm{CF}})\otimes\Z_{(2)},\qquad L:\A^{0}\longrightarrow\Zcal^{\mathrm{cf}},\quad L(w)\coloneqq[w].
\]
For an admissible index $\kk=(k_{1},\ldots,k_{d})$, set 
\begin{equation}
\zeta^{\mathrm{cf}}(\kk)\coloneqq(-1)^{d}L(\word(\kk)).\label{eq:Lmzv}
\end{equation}
Then the confluence relations \cite[Theorem 21]{HSconfluence} give
a $\Z_{(2)}$-linear evaluation map 
\begin{equation}
\Zcal^{\mathrm{cf}}\longrightarrow\R,\qquad\zeta^{\mathrm{cf}}(\kk)\longmapsto\zeta(\kk).\label{eq:numerical-evaluation}
\end{equation}

\begin{thm}
\label{thm:main}Let $\kk\in\Ibb_{N}$. Then $\zeta^{\mathrm{cf}}(\kk)$
is a $\Z_{(2)}$-linear combination of $\zeta^{\mathrm{cf}}(\bb)$
with $\bb\in\Ibb^{\mathrm{H}}_{N}$.
\end{thm}

Applying the evaluation map \eqref{eq:numerical-evaluation} to \cref{thm:main}
gives the following corollary.
\begin{cor}
\label{cor:ordinary-mzv}\Cref{conj:hirose-sato} holds.
\end{cor}

\subsection{Motivic consequences\label{sec:motivic-consequences}}

Furusho characterizes Drinfeld associators and $\mathrm{GRT}_{1}$
by confluence relations \cite[Theorem 1 and Corollary 2]{FurushoConfluence}.
Since the motivic Drinfeld associator satisfies the pentagon equation
\cite[Section 25.9]{AndreMotifs}, these relations hold for motivic
MZVs. \Cref{thm:main} and Brown's Hoffman basis theorem \cite[Theorem 1.1]{BrownMT}
therefore imply that confluence relations or equivalently associator
relations generate all rational relations among motivic MZVs. This
identifies the de Rham motivic Galois group of mixed Tate motives
over $\Z$ with $\mathrm{GRT}$, yielding an alternative proof of
the Deligne--Drinfeld conjecture.

A proof of the Deligne--Drinfeld conjecture was previously presented
in the OpenAI Math Release \cite{OpenAIDD}. A simplified version
of this proof was subsequently given by Willwacher \cite{WillwacherDD}.
Both the OpenAI proof and the present argument use reduction modulo
$2$. 

\subsection{Outline of the proof of the main theorem\label{sec:proof-outline}}

For an admissible index $\kk=(k_{1},\ldots,k_{d})$, put
\[
\odd(\kk)\coloneqq\#\{i:k_{i}\text{ is odd}\},\qquad\depth(\kk)\coloneqq d.
\]

The key object is $R(\kk)\coloneqq\pi\mathscr{D}_{r}\Delta T^{01}V(\kk)$
for a nonempty admissible index $\kk=(k_{1},\ldots,k_{d})$ with $r\coloneqq\odd(\kk)$.
Here $V(\kk)$ is a word encoding $(\lceil k_{1}/2\rceil,\ldots,\lceil k_{d}/2\rceil)$,
with auxiliary labels $z_{0},\ldots,z_{r}$ recording the positions
of the odd entries of $\kk$. The map $T^{01}$ converts it into a
linear combination of words without $e_{0}$, and $\Delta$ replaces
each letter $e_{a}$ by $e_{a}e_{0}$. The operator $\mathscr{D}_{r}$
applies the algebraic differential operators $\partial_{z_{0},0},\ldots,\partial_{z_{r-1},0}$
in this order, and $\pi$ replaces all auxiliary letters by $e_{z}$.
The precise construction is given in \cref{sec:construction}.

Write $\A^{0}_{N}$ for the $\Z$-module spanned by the words of length
$N$ in $\A^{0}$. Put
\begin{align*}
\Hcal_{r,N} & \coloneqq\sum_{\substack{\bb\in\Ibb^{\mathrm{H}}_{N},\,\\
\odd(\bb)\le r
}
}\Z\,\word(\bb),\\
\mathcal{F}_{r,N} & \coloneqq\sum_{\substack{\bb\in\Ibb_{N}\\
\odd(\bb)\le r
}
}\Z\,\word(\bb).
\end{align*}
The map $\operatorname{cf}'$, defined in \cref{sec:confluence},
is a duality-modified confluence relation map. We estimate the two
terms of
\[
\operatorname{cf}'\bigl(R(\kk)\bigr)=\tau_{\infty}\bigl(\lambda(R(\kk))\bigr)-\lambda\bigl(\varphi_{\sh}(R(\kk))\bigr)\in\mathcal{I}_{\mathrm{CF}}.
\]
The following proposition is proved in \cref{sec:proof-first-term,sec:proof-second-term}.
\begin{prop}
\label{prop:two-term-estimates}Let $\kk=(k_{1},\ldots,k_{d})$ be
a nonempty admissible index, and put $N\coloneqq|\kk|$ and $r\coloneqq\odd(\kk)$.
Then
\begin{align*}
\textup{(i)}\quad & \tau_{\infty}(\lambda(R(\kk)))\in\Hcal_{r,N}+\mathcal{F}_{r-2,N},\\
\textup{(ii)}\quad & \lambda\bigl(\varphi_{\sh}(R(\kk))\bigr)\equiv\word(\kk)\pmod{2\A^{0}_{N}+\mathcal{F}_{r-2,N}}.
\end{align*}
\end{prop}

This proposition gives a proof of \cref{thm:main} as follows. Fix
$N\in\Z_{\geq0}$. If $\Ibb_{N}=\Ibb^{\mathrm{H}}_{N}$, the claim
is immediate. Define integers $m_{\kk,\bb}$ by
\[
\operatorname{cf}'\bigl(R(\kk)\bigr)=\sum_{\bb\in\Ibb_{N}}m_{\kk,\bb}\word(\bb)\qquad(\kk\in\Ibb_{N}\setminus\Ibb^{\mathrm{H}}_{N}).
\]

Let $M_{N}$ and $M_{N}'$ be the matrices $(m_{\kk,\bb})_{\kk,\bb\in\Ibb_{N}\setminus\Ibb^{\mathrm{H}}_{N}}$
and $(m_{\kk,\bb})_{\kk\in\Ibb_{N}\setminus\Ibb^{\mathrm{H}}_{N},\bb\in\Ibb^{\mathrm{H}}_{N}}$.
Order $\Ibb_{N}\setminus\Ibb^{\mathrm{H}}_{N}$ in nondecreasing order
of $\odd$. By \cref{prop:two-term-estimates}, for $\kk,\bb\in\Ibb_{N}\setminus\Ibb^{\mathrm{H}}_{N}$
we have
\[
m_{\kk,\bb}\equiv\br{\kk=\bb}\pmod 2\qquad\text{if }\odd(\bb)\ge\odd(\kk).
\]
Here $\br P$ equals $1$ if $P$ holds and $0$ otherwise. Thus $M_{N}$
is lower triangular modulo $2$ with all diagonal entries equal to
$1$. Hence $\det M_{N}$ is odd, so $M^{-1}_{N}$ has entries in
$\Z_{(2)}$.

By construction, for $\kk\in\Ibb_{N}\setminus\Ibb^{\mathrm{H}}_{N}$
we have in $\A^{0}_{N}\otimes_{\Z}\Z_{(2)}$
\[
\word(\kk)=\sum_{\bb\in\Ibb_{N}\setminus\Ibb^{\mathrm{H}}_{N}}(M^{-1}_{N})_{\kk,\bb}\operatorname{cf}'\bigl(R(\bb)\bigr)-\sum_{\bb\in\Ibb^{\mathrm{H}}_{N}}(M^{-1}_{N}M_{N}')_{\kk,\bb}\word(\bb).
\]
The first sum vanishes under the map induced by $L$. Hence \eqref{eq:Lmzv}
gives the explicit expression
\[
\zeta^{\mathrm{cf}}(\kk)=-\sum_{\bb\in\Ibb^{\mathrm{H}}_{N}}(-1)^{\depth(\kk)+\depth(\bb)}(M^{-1}_{N}M_{N}')_{\kk,\bb}\,\zeta^{\mathrm{cf}}(\bb).
\]
All coefficients lie in $\Z_{(2)}$, so this proves \cref{thm:main}.

The argument above gives a simple and complete solution to \cite[Problem 1.4]{HMSWdrop1}.

\section{Confluence relations\label{sec:confluence}}

\subsection{The shuffle product and algebraic differential operators\label{sec:words}}

Let $F$ be a finite set disjoint from $\{0,1\}$. We call the elements
of $\{0,1\}\cup F$ labels. Put 
\[
\A_{F}\coloneqq\Z\langle e_{0},e_{1},e_{a}\ (a\in F)\rangle.
\]
Let $\A^{0}_{F}$ be the module spanned by the empty word and words
that neither begin with $e_{0}$ nor end with $e_{1}$.

The shuffle product $\sh$ on $\A_{F}$ is the bilinear operator defined
by $1\sh w\coloneqq w$, $w\sh1\coloneqq w$ and
\[
(e_{a}u)\sh(e_{b}v)\coloneqq e_{a}\bigl(u\sh e_{b}v\bigr)+e_{b}\bigl(e_{a}u\sh v\bigr).
\]
It is the sum of all interleavings that preserve the order of the
letters within each input word. Associativity and commutativity follow
from this description. The shuffle product is graded by word length.
\begin{defn}
\label{def:deletion} For distinct labels $\alpha,\beta$, define
the algebraic differential operator on $\A_{F}$ by
\[
\partial_{\alpha,\beta}(e_{a_{1}}\cdots e_{a_{n}})\coloneqq\sum^{n}_{i=1}\bigl(\br{\{a_{i},a_{i+1}\}=\{\alpha,\beta\}}-\br{\{a_{i-1},a_{i}\}=\{\alpha,\beta\}}\bigr)e_{a_{1}}\cdots\widehat{e_{a_{i}}}\cdots e_{a_{n}}
\]
where $(a_{0},a_{n+1})\coloneqq(0,1)$.
\end{defn}

\subsection{The map $\varphi_{\otimes,F}$\label{subsec:tensor-map}}

Set
\begin{align*}
\Wcal^{(0)}_{F} & \coloneqq\Z\oplus\sum_{a\in F}e_{a}\Z\langle e_{0},e_{b}\ (b\in F)\rangle,\\
\Wcal^{(1)}_{F} & \coloneqq\Z\oplus\sum_{a\in F}\Z\langle e_{1},e_{b}\ (b\in F)\rangle e_{a}.
\end{align*}
These are the submodules of $\A^{0}_{F}$ spanned by words without
$e_{1}$ and $e_{0}$, respectively.
\begin{defn}
Define $\varphi_{\otimes,F}:\A^{0}_{F}\to\A^{0}\otimes\Wcal^{(0)}_{F}$
by 
\begin{equation}
\varphi_{\otimes,F}(w)\coloneqq\sum_{m\ge0}\ \sum_{b_{1},\ldots,b_{m}\in\{0\}\cup F}\Const(\partial_{1,b_{1}}\cdots\partial_{1,b_{m}}w)\otimes e_{b_{1}}\cdots e_{b_{m}}\label{eq:Psi}
\end{equation}
where the map $\Const$ sends every word containing a letter $e_{a}$
with $a\in F$ to $0$ and fixes every word in $e_{0},e_{1}$ alone.
The operators act from right to left. The sum is finite because each
derivative decreases the length by one.
\end{defn}

By \cref{lem:calculus}, the first tensor factors in \eqref{eq:Psi}
lie in $\A^{0}$ because the algebraic differential operators preserve
admissible words. There are no terms whose second factor begins with
$e_{0}$, since their coefficients are of the form $\Const\partial_{1,0}(\cdots)=0$.
Thus the stated codomain is correct.

With the convention $(u\otimes v)e_{a}\coloneqq u\otimes ve_{a}$,
the definition gives:
\begin{equation}
\varphi_{\otimes,F}(w)=\Const(w)\otimes1+\varphi_{\otimes,F}(\partial_{1,0}w)e_{0}+\sum_{a\in F}\varphi_{\otimes,F}(\partial_{1,a}w)e_{a}.\label{eq:Psi-rec}
\end{equation}

\subsection{Specialization to one auxiliary letter\label{subsec:confluence-specialization}}

For $F=\{z\}$, write
\[
\A_{z}\coloneqq\A_{\{z\}},\qquad\A^{0}_{z}\coloneqq\A^{0}_{\{z\}},\qquad\varphi_{\otimes}\coloneqq\varphi_{\otimes,\{z\}}.
\]
Define
\[
\varphi_{\sh}\coloneqq\operatorname{mult}_{\sh}\circ\varphi_{\otimes}:\A^{0}_{z}\longrightarrow\A^{0}_{z}
\]
where $\operatorname{mult}_{\sh}$ is the $\mathbb{Z}$-linear map
sending $u\otimes v$ to $u\sh v$.

Let $\lambda:\A^{0}_{z}\to\A^{0}$ be the map defined in \cite[Definition 19]{HSconfluence}.
Its construction gives a shuffle algebra homomorphism over $\Z$ preserving
word length. Write $\left.w\right|_{z\to1}$ for the word obtained
from $w$ by replacing $e_{z}$ with $e_{1}$. Set
\[
\A^{-2}_{z}\coloneqq\Z\oplus e_{1}\A_{z}e_{0}\oplus e_{z}\A_{z}e_{0}.
\]
Then 
\begin{equation}
\lambda(w)=\left.w\right|_{z\to1}\quad(w\in\A^{-2}_{z}),\qquad\lambda(e_{z})=0.\label{eq:lambda-B}
\end{equation}

Define the $\Z$-linear map
\[
\operatorname{cf}:\A^{0}_{z}\longrightarrow\A^{0},\qquad\operatorname{cf}(w)\coloneqq\lambda\bigl(w-\varphi_{\sh}(w)\bigr).
\]
By \cite[Definition 22 and Remark 23]{HSconfluence}, the submodule
introduced in \cref{sec:statement} is
\begin{equation}
\mathcal{I}_{\mathrm{CF}}=\{\operatorname{cf}(w):w\in\A^{0}_{z}\}.\label{eq:confluence-ideal}
\end{equation}

Let $\tau_{\infty}$ be the anti-automorphism defined by $\tau_{\infty}(e_{0})\coloneqq-e_{1}$
and $\tau_{\infty}(e_{1})\coloneqq-e_{0}$. The duality relations
satisfy
\begin{equation}
v-\tau_{\infty}v\in\mathcal{I}_{\mathrm{CF}}\qquad(v\in\A^{0})\label{eq:duality}
\end{equation}
by \cite[Theorem 30]{HSconfluence}.

For the reduction, we use the $\Z$-linear map $\operatorname{cf}':\A^{0}_{z}\longrightarrow\A^{0}$
defined by
\[
\operatorname{cf}'(w)\coloneqq\tau_{\infty}\bigl(\lambda(w)\bigr)-\lambda\bigl(\varphi_{\sh}(w)\bigr).
\]
By \eqref{eq:confluence-ideal} and \eqref{eq:duality}, it satisfies
\[
\operatorname{cf}'(w)=\operatorname{cf}(w)+\tau_{\infty}\bigl(\lambda(w)\bigr)-\lambda(w)\in\mathcal{I}_{\mathrm{CF}}.
\]

\section{Identities for algebraic differential operators\label{sec:operators}}
\begin{lem}
\label{lem:calculus}The algebraic differential operators preserve
$\A^{0}_{F}$ and, for $a\in F$, satisfy the following identities:
\begin{align}
\partial_{1,0}u=\partial_{1,a}u & =0\quad(u\in\A^{0}),\label{eq:kill-A0}\\
\Const\partial_{1,0}u=\Const(\partial_{a,0}+\partial_{1,a})u & =0\quad(u\in\A^{0}_{F}).\label{eq:Const-identities}
\end{align}
The identities
\begin{align}
[\partial_{\alpha,\beta},\partial_{\alpha,\gamma}+\partial_{\beta,\gamma}] & =0,\label{eq:comm0}\\{}
[\partial_{\alpha,\beta},\partial_{\gamma,\delta}] & =0\label{eq:commb}
\end{align}
hold whenever the labels in each identity are pairwise distinct. Here
$[P,Q]\coloneqq PQ-QP$.
\end{lem}

\begin{proof}
A deletion can produce an initial $e_{0}$ only by removing the first
letter. If the next letter is $e_{0}$, the two neighboring labels
are both $0$, so the coefficient is $0$. The argument for a terminal
$e_{1}$ is analogous. Thus every algebraic differential operator
preserves $\A^{0}_{F}$.

The identities \eqref{eq:kill-A0} follow directly from the definition.

Since $\partial_{1,0}$ does not delete auxiliary letters, it also
follows that $\Const\partial_{1,0}=0$. The expression $\Const(\partial_{a,0}+\partial_{1,a})$
can be nonzero only when the word contains exactly one auxiliary letter,
namely $e_{a}$. Each neighboring label is $0$ or $1$, so the combined
deletion coefficient of the two operators is $1-1=0$. This proves
\eqref{eq:Const-identities}.

For the commutator identities, terms deleting two nonadjacent positions
cancel. Consider adjacent letters $e_{p},e_{q}$ with outer neighboring
labels $l,r$, including the fixed endpoint labels when necessary.
For $(p,q)=(\alpha,\beta)$, the contributions from deleting both
positions to
\[
[\partial_{\alpha,\beta},\partial_{\alpha,\gamma}]\quad\text{and}\quad[\partial_{\alpha,\beta},\partial_{\beta,\gamma}]
\]
are, respectively,
\[
\br{r=\gamma}-\br{l=\gamma}\quad\text{and}\quad\br{l=\gamma}-\br{r=\gamma}.
\]
The same cancellation holds for the other ordered pairs of distinct
labels in $\{\alpha,\beta,\gamma\}$; all remaining pairs give zero.
This proves \eqref{eq:comm0}. Equation \eqref{eq:commb} follows
by a direct calculation from the definition.
\end{proof}

\section{The doubling homomorphism\label{sec:doubling}}

Define the doubling homomorphism $\Delta:\A_{F}\to\A_{F}$ with respect
to concatenation by
\[
\Delta(e_{a})\coloneqq e_{a}e_{0}\qquad(a\in\{0,1\}\cup F).
\]

\begin{lem}
\label{lem:double-deletions}For $w\in\A^{0}_{F}$ and $a\in F$,
we have 
\begin{gather}
\partial_{1,a}\Delta w=0,\qquad\Const(\partial_{1,0}\Delta w)=0,\label{eq:double-basic}\\
\partial^{2}_{1,0}\Delta w\equiv\Delta(\partial_{1,0}w),\qquad\partial_{1,a}\partial_{1,0}\Delta w\equiv\Delta(\partial_{1,a}w)\pmod 2.\label{eq:double-second}
\end{gather}
\end{lem}

\begin{proof}
By linearity, assume that $w$ is a nonempty word, the empty case
being immediate. Since no $e_{1}$ is adjacent to an auxiliary letter
in $\Delta w$, we have $\partial_{1,a}\Delta w=0$. The constant-term
identity follows because $\partial_{1,0}$ cannot delete auxiliary
letters; if $w$ has none, apply \eqref{eq:kill-A0} to $\Delta w\in\A^{0}$.
This proves \eqref{eq:double-basic}.

For \eqref{eq:double-second}, first suppose that $w$ contains no
$e_{0}$, and work modulo $2$. The first $\partial_{1,0}$ deletes
only inserted zeros, since every original letter has $e_{0}$ on both
sides, including the left boundary. In $\partial^{2}_{1,0}\Delta w$,
terms deleting two inserted zeros cancel by reversing the deletion
order. Each remaining term deletes an adjacent $e_{1}e_{0}$ or $e_{0}e_{1}$.
At $e_{1}e_{0}e_{1}$, the middle zero has coefficient $0$ in the
first deletion, but we may add both pair deletions since their results
agree and cancel modulo $2$. For each original $e_{1}$ not at the
start, the left and right deletions now give the same word and cancel.
Only an initial $e_{1}e_{0}$ remains, if present, yielding $\Delta(\partial_{1,0}w)$
modulo $2$.

If the second operator is $\partial_{1,a}$, only an inserted zero
between original letters $e_{1}$ and $e_{a}$, or in a final block
$e_{a}e_{0}$, can contribute. The second deletion removes either
adjacent original letter in the former case and the final $e_{a}$
in the latter. These are exactly the interior and right-boundary contributions
to $\partial_{1,a}w$, followed by $\Delta$.

For a general word $w$, replace each $e_{0}$ by a fresh auxiliary
letter $e_{0'}$, obtaining $\widetilde{w}$. Let $\pi_{0}$ be the
substitution $e_{0'}\mapsto e_{0}$, fixing all other letters; the
boundary labels remain $0,1$. The definition of the operators gives
\[
\partial_{1,0}\pi_{0}=\pi_{0}(\partial_{1,0}+\partial_{1,0'}),\qquad\partial_{1,a}\pi_{0}=\pi_{0}\partial_{1,a},\qquad\pi_{0}\Delta=\Delta\pi_{0}.
\]

Apply the case already proved to $\widetilde{w}$, with auxiliary
set $F\cup\{0'\}$. Since $\partial_{1,0'}\Delta\widetilde{w}=0$,
we obtain 
\begin{align*}
\partial^{2}_{1,0}\Delta w & =\pi_{0}\bigl(\partial^{2}_{1,0}+\partial_{1,0'}\partial_{1,0}\bigr)\Delta\widetilde{w}\\
 & \equiv\pi_{0}\Delta(\partial_{1,0}+\partial_{1,0'})\widetilde{w}=\Delta(\partial_{1,0}w)\pmod 2,\\
\partial_{1,a}\partial_{1,0}\Delta w & =\pi_{0}\partial_{1,a}\partial_{1,0}\Delta\widetilde{w}\\
 & \equiv\pi_{0}\Delta\partial_{1,a}\widetilde{w}=\Delta(\partial_{1,a}w)\pmod 2.\qedhere
\end{align*}
\end{proof}

\begin{lem}
\label{lem:doubling}For $w\in\A^{0}_{F}$, 
\begin{equation}
\varphi_{\otimes,F}(\Delta w)\equiv(\Delta\otimes\Delta)\varphi_{\otimes,F}(w)\pmod 2.\label{eq:doubling}
\end{equation}
\end{lem}

\begin{proof}
Applying \eqref{eq:Psi-rec} twice and using \cref{lem:double-deletions},
we obtain 
\[
\varphi_{\otimes,F}(\Delta w)\equiv\Delta(\Const w)\otimes1+\varphi_{\otimes,F}(\Delta\partial_{1,0}w)e^{2}_{0}+\sum_{a\in F}\varphi_{\otimes,F}(\Delta\partial_{1,a}w)e_{a}e_{0}\pmod 2.
\]
 Induction on word length, together with \eqref{eq:Psi-rec}, proves
\eqref{eq:doubling}.
\end{proof}

\section{The maps \texorpdfstring{$T^{10}$ and $T^{01}$}{T\string^10 and T\string^01}\label{sec:inverse-maps}}
\begin{lem}
\label{lem:intertwine}For distinct labels $\alpha,\beta$, we have
on $\A^{0}_{F}$
\begin{equation}
\varphi_{\otimes,F}\partial_{\alpha,\beta}=(\id\otimes\partial_{\alpha,\beta})\varphi_{\otimes,F}.\label{eq:Psi-D}
\end{equation}
\end{lem}

\begin{proof}
For $b\in\{0\}\cup F$, the operator $\partial_{1,b}$ acts on $\Wcal^{(0)}_{F}$
by removing a terminal $e_{b}$ and gives zero otherwise. Thus \eqref{eq:Psi}
proves the claim when one of the labels is $1$. We will use this
commutativity repeatedly below without further comment.

Now suppose that $\alpha,\beta\in\{0\}\cup F$, and use induction
on the length of a word $w\in\A^{0}_{F}$, the empty case being immediate.
A tensor in $\A^{0}\otimes\Wcal^{(0)}_{F}$ is determined by its component
with empty second factor and its images under $\id\otimes\partial_{1,c}$
for all $c\in\{0\}\cup F$. We compare these components and images.

By definition, the left-hand empty-word component is $\Const(\partial_{\alpha,\beta}w)$.
On the right, $\partial_{\alpha,\beta}$ can produce the empty word
only from $e_{\beta}$ and $e_{\alpha}$, with coefficients $-\br{\alpha=0}$
and $-\br{\beta=0}$, respectively. Thus, the right-hand empty-word
component is
\[
-\br{\alpha=0}\Const(\partial_{1,\beta}w)-\br{\beta=0}\Const(\partial_{1,\alpha}w)=\Const(\partial_{\alpha,\beta}w),
\]
where the equality follows directly from the definition of the operators.

For $c\in\{0\}\cup F$, the commutator identities \eqref{eq:comm0}
and \eqref{eq:commb} give
\begin{equation}
[\partial_{1,c},\partial_{\alpha,\beta}]=(\br{c=\beta}-\br{c=\alpha})[\partial_{1,\alpha},\partial_{1,\beta}].\label{eq:intertwine-commutator}
\end{equation}
We have
\begin{align*}
(\id\otimes\partial_{1,c})\varphi_{\otimes,F}(\partial_{\alpha,\beta}w) & =\varphi_{\otimes,F}(\partial_{\alpha,\beta}\partial_{1,c}w)+\varphi_{\otimes,F}([\partial_{1,c},\partial_{\alpha,\beta}]w)\\
 & =(\id\otimes\partial_{\alpha,\beta}\partial_{1,c})\varphi_{\otimes,F}(w)+(\id\otimes[\partial_{1,c},\partial_{\alpha,\beta}])\varphi_{\otimes,F}(w)\\
 & =(\id\otimes\partial_{1,c}\partial_{\alpha,\beta})\varphi_{\otimes,F}(w),
\end{align*}
where the second equality follows from the induction hypothesis for
the first term and \eqref{eq:intertwine-commutator} for the second
term. This completes the proof.
\end{proof}

Define $\Z$-linear maps $T^{10}:\Wcal^{(1)}_{F}\to\Wcal^{(0)}_{F}$
and $T^{01}:\Wcal^{(0)}_{F}\to\Wcal^{(1)}_{F}$ recursively by $T^{10}(1)\coloneqq1$
and $T^{01}(1)\coloneqq1$, and, for nonempty words $w\in\Wcal^{(1)}_{F}$
and $v\in\Wcal^{(0)}_{F}$, by
\begin{align*}
T^{10}(w) & \coloneqq\sum_{b\in\{0\}\cup F}T^{10}(\partial_{1,b}w)e_{b},\\
T^{01}(v) & \coloneqq-\sum_{a\in\{1\}\cup F}e_{a}T^{01}(\partial_{a,0}v).
\end{align*}
Induction on word length verifies the stated codomains.

Since $\Const(w)=0$ for every nonempty word $w\in\Wcal^{(1)}_{F}$,
\eqref{eq:Psi-rec} and induction give
\begin{equation}
\varphi_{\otimes,F}(w)=1\otimes T^{10}(w)\qquad(w\in\Wcal^{(1)}_{F}).\label{eq:T-def}
\end{equation}

\begin{lem}
\label{lem:T-inverse} 
\[
T^{10}T^{01}=\id_{\Wcal^{(0)}_{F}},\qquad T^{01}T^{10}=\id_{\Wcal^{(1)}_{F}}.
\]
In particular, for $V\in\Wcal^{(0)}_{F}$, 
\begin{equation}
\varphi_{\otimes,F}(T^{01}(V))=1\otimes V.\label{eq:S}
\end{equation}
\end{lem}

\begin{proof}
By \cref{lem:intertwine} and \eqref{eq:T-def}, $T^{10}$ commutes
with $\partial_{a,0}$ for $a\in\{1\}\cup F$. For a nonempty word
$w\in\Wcal^{(1)}_{F}$, induction on its length gives
\[
T^{01}T^{10}(w)=-\sum_{a\in\{1\}\cup F}e_{a}T^{01}T^{10}(\partial_{a,0}w)=-\sum_{a\in\{1\}\cup F}e_{a}\partial_{a,0}w=w.
\]
The last equality holds because, on $\Wcal^{(1)}_{F}$, $\partial_{a,0}$
removes an initial $e_{a}$ with a minus sign and gives zero otherwise.
The empty case is immediate. Hence $T^{01}T^{10}=\id_{\Wcal^{(1)}_{F}}$.
In each fixed word length, the two spaces are free $\Z$-modules of
equal finite rank. Since both maps preserve length, it follows that
$T^{10}T^{01}=\id_{\Wcal^{(0)}_{F}}$ as well.
\end{proof}

\section{Construction of \texorpdfstring{$R(\protect\kk)$}{R(k)}\label{sec:construction}}

For any finite set $F$ of auxiliary labels, let $\pi:\A_{F}\to\A_{z}$
be the algebra homomorphism defined by $\pi(e_{a})\coloneqq e_{z}$
for $a\in F$, and set $\pi(e_{0})\coloneqq e_{0}$ and $\pi(e_{1})\coloneqq e_{1}$.

Now fix a nonempty admissible index $\kk=(k_{1},\ldots,k_{d})$. Let
$j_{1}<\cdots<j_{r}$ be the positions of the odd entries in $\kk$.
Put 
\[
K_{i}\coloneqq k_{i}+\br{k_{i}\text{ is odd}}\in2\Z_{\geq1},\qquad a_{i}\coloneqq K_{i}/2\in\Z_{\geq1}.
\]
Let $F\coloneqq\{z_{0},\ldots,z_{r}\}$ be a set of distinct auxiliary
labels. With the conventions $j_{0}\coloneqq0$ and $j_{r+1}\coloneqq d$,
define 
\[
c_{i}\coloneqq z_{s}\quad(j_{s}<i\le j_{s+1}),\qquad0\le s\le r.
\]
The last interval may be empty when $j_{r}=d$. If $r=0$, all the
labels $c_{i}$ are $z_{0}$. Set

\[
V(\kk)\coloneqq e_{c_{1}}e^{a_{1}-1}_{0}\cdots e_{c_{d}}e^{a_{d}-1}_{0}\in\Wcal^{(0)}_{F},\qquad\mathscr{D}_{r}\coloneqq\partial_{z_{r-1},0}\cdots\partial_{z_{1},0}\partial_{z_{0},0}.
\]
For $r=0$, take $\mathscr{D}_{0}\coloneqq\id$. For example, 
\[
V(8,5,8,10)=e_{z_{0}}e^{3}_{0}\,e_{z_{0}}e^{2}_{0}\,e_{z_{1}}e^{3}_{0}\,e_{z_{1}}e^{4}_{0}.
\]
Finally, define
\[
R(\kk)\coloneqq\pi\mathscr{D}_{r}\Delta T^{01}V(\kk)\in\A^{0}_{z}.
\]

\section{Proof of \texorpdfstring{\cref{prop:two-term-estimates}}{Proposition \getrefnumber{prop:two-term-estimates}}
(i)\label{sec:proof-first-term}}
\begin{lem}
\label{lem:excess}Let $\kk=(k_{1},\ldots,k_{d})$ be an index with
$k_{i}\ge2$ for all $i$. Put $s\coloneqq|\kk|-2\depth(\kk)\in\mathbb{Z}_{\geq0}$.
Then 
\[
\word(\kk)\in\Hcal_{s,|\kk|}+\mathcal{F}_{s-2,|\kk|}.
\]
\end{lem}

\begin{proof}
By assumption, $\odd(\kk)\leq s$, with equality if and only if $\kk$
is a Hoffman index. Since $s-\odd(\kk)$ is even, $\word(\kk)\in\mathcal{F}_{s-2,|\kk|}$
if $\kk$ is not a Hoffman index, and $\word(\kk)\in\Hcal_{s,|\kk|}$
otherwise.
\end{proof}

We now prove part (i) of \cref{prop:two-term-estimates}:
\[
\tau_{\infty}(\lambda(R(\kk)))\in\Hcal_{r,N}+\mathcal{F}_{r-2,N}
\]
where $r\coloneqq\odd(\kk)$ and $N\coloneqq|\kk|$. Recall that $R(\kk)=\pi\mathscr{D}_{r}\Delta T^{01}V(\kk)$.
\begin{proof}[Proof of \cref{prop:two-term-estimates} (i)]
Put $n\coloneqq|V(\kk)|=(N+r)/2$. Every word in $\Delta T^{01}V(\kk)$
has the form
\begin{equation}
e_{b_{1}}e_{0}\,e_{b_{2}}e_{0}\cdots e_{b_{n}}e_{0},\qquad b_{i}\in\{1\}\cup F.\label{eq:alternating}
\end{equation}
Deleting a nonzero letter between two zeros has coefficient $1-1=0$.
Thus every word $w$ occurring in $R(\kk)$ contains no $e_{0}e_{0}$
and has at least $n-r\geq1$ occurrences of $e_{0}$.

The operators in $\mathscr{D}_{r}$ do not delete $e_{1}$. If an
input word of the form \eqref{eq:alternating} contains $e_{1}$,
its suffix after the last $e_{1}$ has $m$ auxiliary letters and
$m+1$ zeros. Only operators indexed by these auxiliary labels can
delete letters in this suffix, and each is used at most once, so at
most $m$ letters are deleted and a zero remains. Thus $w$ has only
$e_{z}$'s after its last $e_{0}$. Shuffle regularization moves these
letters before the last $e_{0}$, preserving the number of zeros and
creating no $e_{0}e_{0}$. By \eqref{eq:lambda-B}, $\lambda(w)$
is an integral linear combination of admissible words with these properties.

Since $\tau_{\infty}$ interchanges $e_{0}$ and $e_{1}$ up to sign,
every index $\avec$ occurring in $\tau_{\infty}(\lambda(w))$ has
all entries at least $2$ and satisfies $\depth(\avec)\ge n-r$. Hence
\[
|\avec|-2\depth(\avec)\le(2n-r)-2(n-r)=r.
\]
By \Cref{lem:excess}, $\word(\avec)\in\Hcal_{r,N}+\mathcal{F}_{r-2,N}$,
which completes the proof.
\end{proof}

\section{Proof of \texorpdfstring{\cref{prop:two-term-estimates}}{Proposition \getrefnumber{prop:two-term-estimates}}
(ii)\label{sec:proof-second-term}}

For any finite set $F$ of auxiliary labels, we have 
\begin{equation}
\varphi_{\otimes}\pi=(\id\otimes\pi)\varphi_{\otimes,F}.\label{eq:collapse}
\end{equation}
Indeed, $\Const\pi=\Const$, $\partial_{1,0}\pi=\pi\partial_{1,0}$,
and $\partial_{1,z}\pi=\pi\sum_{a\in F}\partial_{1,a}$, so the sums
over sequences of derivatives in \eqref{eq:Psi} correspond term by
term.
\begin{lem}
\label{prop:colored} Let $V\in\Wcal^{(0)}_{F}$, and put $S\coloneqq T^{01}(V)\in\Wcal^{(1)}_{F}$
as in \eqref{eq:S}. For any $a_{1},\ldots,a_{s}\in F$, 
\[
\varphi_{\sh}\!\left(\pi\partial_{a_{1},0}\cdots\partial_{a_{s},0}\Delta S\right)\equiv\pi\partial_{a_{1},0}\cdots\partial_{a_{s},0}\Delta V\pmod{2\A^{0}_{z}}.
\]
\end{lem}

\begin{proof}
By \eqref{eq:collapse}, \eqref{eq:Psi-D}, \cref{lem:doubling},
and \eqref{eq:S}, we have 
\[
\varphi_{\otimes}\pi\partial_{a_{1},0}\cdots\partial_{a_{s},0}\Delta S\equiv1\otimes\pi\partial_{a_{1},0}\cdots\partial_{a_{s},0}\Delta V\pmod 2.
\]
Applying the shuffle multiplication map gives the claim.
\end{proof}

Fix a nonempty admissible index $\kk=(k_{1},\ldots,k_{d})$, and use
the notation of \cref{sec:construction}. For any index $\bb=(b_{1},\ldots,b_{m})$,
write
\[
\word_{z}(\bb)\coloneqq e_{z}e^{b_{1}-1}_{0}\cdots e_{z}e^{b_{m}-1}_{0}.
\]

\begin{lem}
\label{lem:extraction} Let $N\coloneqq|\kk|$. Then
\[
\pi\mathscr{D}_{r}\Delta V(\kk)\equiv(-1)^{r}\word_{z}(\kk)\pmod{\mathcal{G}_{r-2,N}}
\]
where $\mathcal{G}_{r-2,N}$ is the $\Z$-submodule generated by the
words $\word_{z}(\bb)$ for all admissible indices $\bb$ of weight
$N$ with $\odd(\bb)\le r-2$.
\end{lem}

\begin{proof}
For a word $w\in\Wcal^{(0)}_{F}$ with $\pi(w)=\word_{z}(\bb)$, let
$\odd(w)$ denote the number of odd entries of $\bb$, which need
not be admissible; set $\odd(1)\coloneqq0$. For $m\in\Z$, let $G_{m}\subset\Wcal^{(0)}_{F}$
be the submodule generated by words $w$ with $\odd(w)\le m$. Since
each deletion increases the number of odd entries by at most one,
\[
\partial_{z_{i},0}(G_{s})\subset G_{s+1}.
\]

Let $U_{m}$ ($m=0,\ldots,r$) be the word obtained from
\[
\Delta V(\kk)=e_{c_{1}}e^{K_{1}-1}_{0}\cdots e_{c_{d}}e^{K_{d}-1}_{0}
\]
by decreasing $K_{j_{1}},\ldots,K_{j_{m}}$ by one. Then $\odd(U_{m})=m$.
In $\partial_{z_{m},0}(U_{m})$, the contributions from deleting $e_{0}$
between consecutive occurrences of $e_{z_{m}}$ cancel, leaving $-U_{m+1}$
as the only term outside $G_{m-1}$:
\[
\partial_{z_{m},0}(U_{m})\equiv-U_{m+1}\pmod{G_{m-1}}
\]
for $0\leq m<r$. Applying the operators successively gives
\[
\mathscr{D}_{r}\Delta V(\kk)=\mathscr{D}_{r}U_{0}\equiv(-1)^{r}U_{r}\pmod{G_{r-2}}.
\]

All resulting words in $\mathscr{D}_{r}\Delta V(\kk)$ begin with
an auxiliary letter and have length $N$. If $k_{d}$ is even, the
terminal block is untouched. If $k_{d}$ is odd, the terminal zero
run initially has length $k_{d}\ge3$ and can be shortened only by
the last operator, which deletes at most one zero. Thus every resulting
index is admissible. Applying $\pi$ now gives
\[
\pi\mathscr{D}_{r}\Delta V(\kk)\equiv(-1)^{r}\pi U_{r}=(-1)^{r}\word_{z}(\kk)\pmod{\mathcal{G}_{r-2,N}}.\qedhere
\]
\end{proof}

\begin{proof}[Proof of \cref{prop:two-term-estimates} (ii)]
Put $N\coloneqq|\kk|$. \Cref{prop:colored} gives
\[
\varphi_{\sh}\bigl(R(\kk)\bigr)\equiv\pi\mathscr{D}_{r}\Delta V(\kk)\pmod{2\A^{0}_{z}}.
\]
Applying the map $\lambda$ and using \cref{lem:extraction} and \eqref{eq:lambda-B},
we obtain
\[
\lambda\bigl(\varphi_{\sh}(R(\kk))\bigr)\equiv(-1)^{r}\word(\kk)\pmod{2\A^{0}_{N}+\mathcal{F}_{r-2,N}}.
\]
Since $(-1)^{r}\equiv1\pmod 2$, part (ii) follows.
\end{proof}

\subsection*{Acknowledgements}

The author thanks Shin-ichiro Seki for helpful comments on the manuscript.
This work was supported by JSPS KAKENHI Grant Number JP22K03244.

\subsection*{Use of AI}

The author used ChatGPT (GPT-6 Astra Pro) for mathematical discussions
and assistance with the exposition and English prose of this manuscript.
All mathematical arguments were independently verified by the author,
who reviewed the final manuscript and takes full responsibility for
its content.

\end{document}